\documentclass[11pt, reqno]{amsart}
\usepackage{amsmath, amssymb}
\usepackage{graphicx}
\usepackage[all]{xy}
\usepackage{color,url}
\usepackage{lscape}

\newtheorem{theorem}{Theorem}[section]
\newtheorem{lemma}[theorem]{Lemma}
\newtheorem{proposition}[theorem]{Proposition}
\newtheorem{corollary}[theorem]{Corollary}
\theoremstyle{definition}
\newtheorem{definition}[theorem]{Definition}
\newtheorem{example}[theorem]{Example}
\newtheorem{problem}[theorem]{Problem}
\theoremstyle{remark}
\newtheorem{remark}[theorem]{Remark}
\numberwithin{equation}{section}

\newcommand{\Z}{\mathbb{Z}}
\newcommand{\Q}{\mathbb{Q}}
\newcommand{\R}{\mathbb{R}}
\newcommand{\C}{\mathbb{C}}
\newcommand{\A}{\mathbb{A}}

\newcommand{\bb}{\mathbf{b}}
\newcommand{\cc}{\mathbf{c}}

\renewcommand{\Im}{\operatorname{Im}}

\newcommand{\Hom}{\mathrm{Hom}}

\newcommand{\irr}{\mathrm{irr}}

\newcommand{\SU}{\mathit{SU}}
\newcommand{\SL}{\mathit{SL}}

\newcommand{\RT}{\mathit{RT}}
\newcommand{\res}{\mathrm{res}}
\newcommand{\tr}{\operatorname{tr}}

\newcommand{\ang}[1]{\langle#1\rangle}

\makeatletter
\@namedef{subjclassname@2020}{\textup{2020} Mathematics Subject Classification}
\makeatother

\begin{document}

\title[An algebraic property of Reidemeister torsion]{An algebraic property of Reidemeister torsion}

\author{Teruaki Kitano}
\address{Department of Information Systems Science, Faculty of Science and Engineering, Soka University \\
Tangi-cho 1-236, Hachioji, Tokyo 192-8577 \\
Japan}
\email{kitano@soka.ac.jp}

\author{Yuta Nozaki}
\address{
Graduate School of Advanced Science and Engineering, Hiroshima University \\
1-3-1 Kagamiyama, Higashi-Hiroshima City, Hiroshima, 739-8526 \\
Japan}
\email{nozakiy@hiroshima-u.ac.jp}



\maketitle

\begin{abstract}
For a 3-manifold $M$ and an acyclic $\mathit{SL}(2,\mathbb{C})$-representation $\rho$ of its fundamental group, the $\mathit{SL}(2,\mathbb{C})$-Reidemeister torsion $\tau_\rho(M) \in \mathbb{C}^\times$ is defined.
If there are only finitely many conjugacy classes of irreducible representations, then the Reidemeister torsions are known to be algebraic numbers.
Furthermore, we prove that the Reidemeister torsions are not only algebraic numbers but also algebraic integers for most Seifert fibered spaces and infinitely many hyperbolic 3-manifolds.
Also, for a knot exterior $E(K)$, we discuss the behavior of $\tau_\rho(E(K))$ when the restriction of $\rho$ to the boundary torus is fixed.
\end{abstract}

\setcounter{tocdepth}{1}
\tableofcontents

\section{Introduction}
\label{sec:Intro}
Let $M$ be a connected compact $n$-manifold and let $\rho$ be an acyclic $\SL(2,\C)$-representation, namely the chain complex $C_\ast(M;\C^2_\rho)$ is acyclic.
Then the $\SL(2,\C)$-Reidemeister torsion $\tau_\rho(M) \in \C^\times$ is defined to be the alternative product of determinants (see Section~\ref{subsec:Rtorsion}).
When $\rho$ is not acyclic, we set $\tau_\rho(M)=0$.
Then $\tau_\rho(M)$ defines a $\C$-valued function on the $\SL(2,\C)$-representation variety $R(M)=\Hom(\pi_1(M),\SL(2,\C))$, which factors through the $\SL(2,\C)$-character variety $X(M)$ of $M$.
In this paper, we mainly consider a 3-manifold $M$ and the subspace $X^\irr(M)$ of irreducible characters.
Under the assumption that $X^\irr(M)$ is a finite set, Johnson~\cite{Joh88} defined the torsion polynomial $\sigma_M(t)$ of $M$ by
\[
\sigma_M(t) = \prod_{[\rho] \in X^\irr(M),\ \text{acyclic}}(t-\tau_\rho(M)) \in \C[t]. %
\]
He mentioned that $\sigma_M(t)$ lies in $\Q[t]$ by considering the action of a Galois group.
As a consequence, $\tau_\rho(M)$ is an algebraic number when $X^\irr(M)$ is a finite set.
The first author and Tran~\cite{KiTr} described the torsion polynomial of the Brieskorn homology 3-sphere $\Sigma(p,q,r)$ in terms of the normalized Chebyshev polynomial (see also \cite{Kit17}).
Moreover, we show that $\sigma_{\Sigma(p,q,r)}(t)$ has integral coefficients in Section~\ref{sec:Seifert}.
In other words, $\tau_\rho(\Sigma(p,q,r))$ is not only an algebraic number but also an algebraic integer for every irreducible representation $\rho$.
Recall that $\alpha \in \C$ is called an \emph{algebraic integer} if there is a monic polynomial over $\Z$ such that $\alpha$ is a root of the polynomial.
As a consequence, for the homology 3-sphere obtained by Dehn surgery along the (right-handed) $(p,q)$-torus knot $T_{p,q}$, its Reidemeister torsion $\tau_\rho(S^3_{-1/n}(T_{p,q}))$ is also an algebraic integer since $\Sigma(p,q,pqn+1) = S^3_{-1/n}(T_{p,q})$ for $n>0$.

Also, such a phenomenon was numerically observed for the figure-eight knot $4_1$ in \cite{Kit16N}.
It is worth mentioning that while the Brieskorn homology 3-spheres are not hyperbolic, $S^3_{p/q}(4_1)$ is hyperbolic unless $p/q$ is an integer with $|p/q| \leq 4$ or $p/q=\infty$.
In this paper, we rigorously prove the observation in \cite{Kit16N}.

\begin{theorem}
\label{thm:4_1}
Let $p,q$ be non-zero integers and either $p$ or $q$ is $1$.
For an $\SL(2,\C)$-representation $\rho$ of $\pi_1(S^3_{p/q}(4_1))$, the $\SL(2,\C)$-Reidemeister torsion $\tau_\rho(S^3_{p/q}(4_1))$ is an algebraic integer.
\end{theorem}

\begin{figure}[h]
 \centering
 \includegraphics[width=0.7\textwidth]{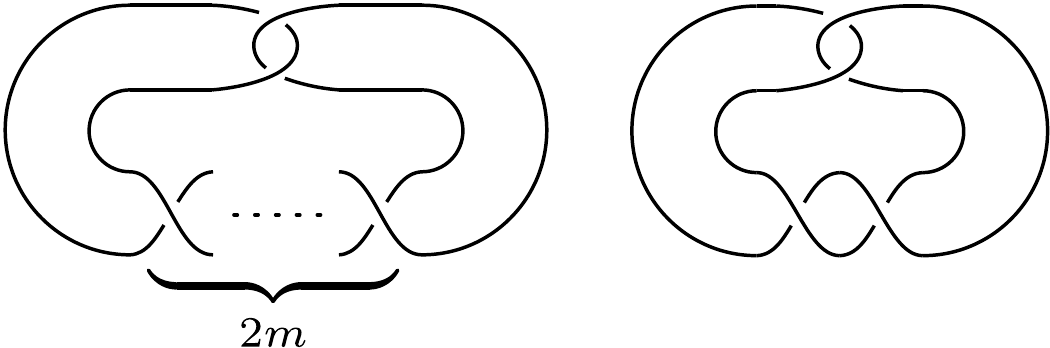}
 \caption{The twist knot $J(2,2m)$ and $J(2,2)=3_1$, where $2m$ half-twists mean $|2m|$ negative half-twists if $m<0$.}
 \label{fig:twist}
\end{figure}

Furthermore, an analogous result holds for the complement of a twist knot $J(2,2m)$ in Figure~\ref{fig:twist} ($m \in \Z$), which suggests that the Reidemeister torsions of the closed 3-manifold $S^3_{p/q}(J(2,2m))$ might be algebraic integers as observed in Examples~\ref{ex:2/3} and \ref{ex:5_2}.
Throughout this paper, for a knot $K$, let $E(K)$ denote the complement of an open tubular neighborhood of $K$.
We also write $\mu$ and $\lambda$ for a meridian and a preferred longitude of $K$, respectively.

\begin{theorem}
\label{thm:twist_knot}
Let $K$ be a twist knot $J(2,2m)$ and let $p=\pm 1$ and $q$ an odd integer.
For an irreducible $\SL(2,\C)$-representation $\rho$ of $\pi_1(E(K))$ satisfying $\rho(\mu^p\lambda^q)=I_2$, the $\SL(2,\C)$-Reidemeister torsion $\tau_\rho(E(K))$ is an algebraic integer.
\end{theorem}

Note that $J(2,2)$ is a trefoil knot $3_1=T_{-3,2}$, and $\tau_\rho(E(3_1))$ was computed by Johnson~\cite{Joh88}.
In fact, he expressed $\tau_\rho(E(T_{p,q}))$ in terms of cosine functions.
In Section~\ref{sec:Seifert}, we see that these are algebraic integers as well.
We give the proofs of Theorems~\ref{thm:4_1} and \ref{thm:twist_knot} in Section~\ref{sec:algebraic_integer}, which is based on the resultant of polynomials and the $A$-polynomial of knots introduced by Cooper, Culler, Gillet, Long, and Shalen~\cite{CCGLS94}.

Here we refer to some related results.
First recall that the Reidemeister torsion $\tau_\rho(E(K))$ coincides with $\Delta_{K,\rho}(1)$, where $\Delta_{K,\rho}(t) \in \C[t^{\pm 1}]$ denotes the twisted Alexander polynomial associated with $\rho$ (see \cite{Wad94}, \cite{Kit96T}).
When $K$ is hyperbolic and $\rho$ is the holonomy representation, Dunfield, Friedl, and Jackson~\cite{DFJ12} called $\Delta_{K,\rho}(t)$ the hyperbolic torsion polynomial.
They observed that the coefficients of hyperbolic torsion polynomial are often algebraic integers.
This observation and Theorems~\ref{thm:4_1} and \ref{thm:twist_knot} suggest that, under some conditions, Reidemeister torsions have an interesting algebraic property that is quite non-trivial from its definition.
Recently, Yoon~\cite{Yoo20B} proved a vanishing identity on the adjoint Reidemeister torsions of 2-bridge knots.
This result also implies that Reidemeister torsions are imposed on strong algebraic restrictions.
Moreover, we observe that the maximal $\SL(2,\C)$-Reidemeister torsion in absolute value is often a Perron number in Examples~\ref{ex:2/3} and \ref{ex:5_2}.
Here, a real algebraic number $\alpha>1$ is called a \emph{Perron number} if it is larger than the absolute values of the Galois conjugates of $\alpha$.
Such numbers appear, for example, in the study of the stretch factor of pseudo-Anosov homeomorphisms (see \cite[Section~14.2.1]{FaMa12}).

\begin{figure}[h]
 \centering
 \includegraphics[width=0.7\textwidth]{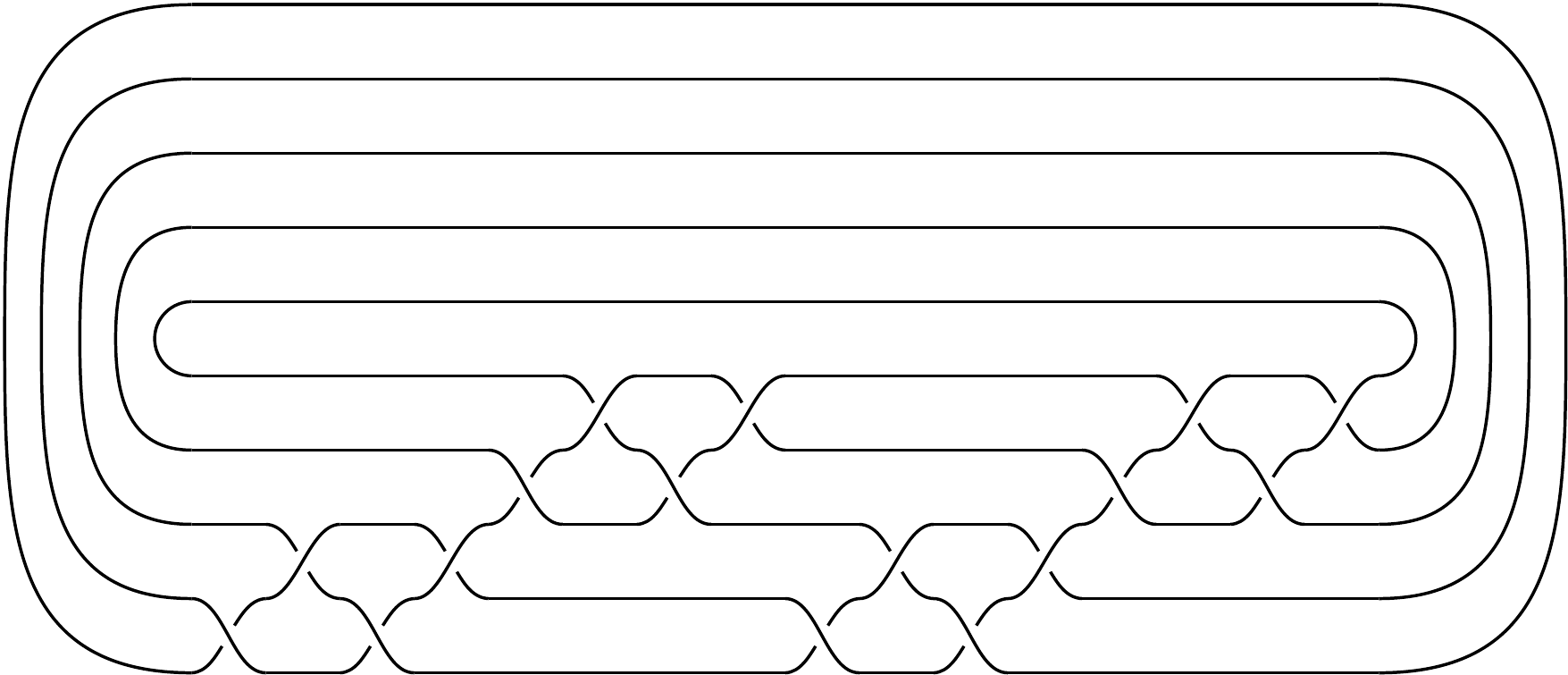}
 \caption{The knot $K_0$.}
 \label{fig:K0}
\end{figure}

We now turn our attention to the case where $X^\irr(M)$ is an infinite set.
For instance, if $M=E(K)$ for a knot $K$ in $S^3$, then $X^\irr(M)$ has a positive dimension.
The first author~\cite{Kit94E} gave the formula $\tau_\rho(E(4_1)) = 2-2\tr\rho(\mu)$, where $\tr\rho(\mu)\neq 2$ for a meridian $\mu$.
Note that $\tau_\rho(E(4_1))$ is no longer an algebraic integer since $\tr\rho(\mu)$ can vary continuously.
On the other hand, $\tau_\rho(E(4_1))$ is determined by the restriction $r(\rho)$ of $\rho$ to the boundary torus.
Here, when $\partial M\neq \emptyset$, we denote by $r$ the regular map $X(M) \to X(\partial M)$ between algebraic sets induced by the inclusion $\partial M \hookrightarrow M$.
It is natural to ask whether the function $\tau_\rho(M)$ on $X(M)$ varies continuously while $r(\rho)$ is fixed.
If $K$ is a 2-bridge knot, then $\dim_\C X(E(K))=1$ and $r^{-1}(\rho_0)$ is a finite set, and thus $\tau_\rho(M)$ cannot vary continuously.
We need to consider the case $\dim_\C X(E(K))\geq 2$.
An easy example of such a knot is the connected sum $K=K_1\sharp K_2$ of non-trivial knots $K_1$ and $K_2$.
However, if $r(\rho)$ is fixed, then $\tau_\rho(E(K))$ depends only on $\tau_\rho(E(K_1))$ and $\tau_\rho(E(K_2))$ because of the multiplicativity of the Reidemeister torsions.
In this paper, we focus on the knot $K_0$ in Figure~\ref{fig:K0}, which is obtained from $4_1\sharp 4_1$ by a construction given by Cooper and Long~\cite{CoLo96}.
Then we can find a family $C$ of representations of $\pi_1(E(K_0))$ and obtain the following result.

\begin{theorem}
\label{thm:vary_conti}
Let $\rho' \in C$.
Then the function $\tau_\rho(E(K_0))$ varies continuously on the preimage $r^{-1}(r(\rho')) \subset X(E(K_0))$.
\end{theorem}

Note that the knot $K_0$ is alternating, hyperbolic, and fibered (see Section~\ref{sec:Looper-Long}).
Theorem~\ref{thm:vary_conti} is motivated by the authors' previous work~\cite{KiNo20}.
They investigated the set
\[
\RT(M)=\{\tau_\rho(M) \mid [\rho] \in X^\irr(M),\ \text{acyclic}\} \subset \C
\]
of all values of the Reidemeister torsion for irreducible representations.
In \cite{KiNo20}, the authors proved that for 2-bridge knots $K_1$ and $K_2$, the set $\RT(\Sigma(K_1,K_2))$ is finite, while $X^\irr(\Sigma(K_1,K_2))$ has positive dimension.
Here the splice $\Sigma(K_1,K_2)$ is the closed 3-manifold $E(K_1)\cup_h E(K_2)$, where $h$ is an orientation-revering homeomorphism sending a meridian (resp.\ longitude) of $K_1$ to a longitude (resp.\ meridian) of $K_2$.
The proof is based on the fact that $\tau_\rho(E(K_j))$ cannot vary continuously if $r_j(\rho)$ is fixed ($j=1,2$) as mentioned before.
In contrast, Theorem~\ref{thm:vary_conti} implies the following consequence.

\begin{corollary}
\label{cor:RT_infinite}
Let $K$ be a $2$-bridge knot such that polynomials $f_C(L,M)$ and $A_K(M,L) \in \Z[L,M]$ have a common zero $(L_0,M_0)$ with $L_0,M_0 \neq 0$ and with $L_0\neq \pm 1$ or $M_0\neq \pm 1$, where $A_K$ is the $A$-polynomial of $K$.
Then the set $\RT(\Sigma(K_0,K))$ is an infinite set.
\end{corollary}

Here, $f_C(L,M) \in \Z[L,M]$ is the polynomial
\begin{align*}
 & L^2 M^{16}-L((M^{32}+1)-4(M^{30}+M^{2})-2(M^{28}+M^{4})+16(M^{26}+M^{6}) \\
 & +13(M^{24}+M^{8})-32(M^{22}+M^{10})-46(M^{20}+M^{12})+20(M^{18}+M^{14}) +70M^{16})+M^{16},
\end{align*}
which is derived from $r(C) \subset X(\partial E(K_0))$.
Note that the condition on $K$ in Corollary~\ref{cor:RT_infinite} is generic enough.
For instance, $3_1$ and $4_1$ satisfy the condition.

\subsection*{Acknowledgments}
This work was supported by JSPS KAKENHI Grant Numbers JP19K03505 and JP20K14317.

\section{Preliminaries}
\label{sec:Preliminaries}
\subsection{$\SL(2,\C)$-Reidemeister torsion and the torsion polynomial}
\label{subsec:Rtorsion}
Let $M$ be a connected compact 3-manifold whose boundary is empty or a disjoint union of tori and let $\rho\colon \pi_1(M) \to \SL(2,\C)$ be a representation, namely a group homomorphism.
Suppose $\rho$ is acyclic, that is, $H_\ast(M;\C^2_\rho)=0$.
We first endow $M$ with a cell decomposition so that $M$ is a CW-complex.
This gives the cellular chain complex $\{C_\ast(M;\C^2_\rho), \partial_\ast\}$ with an (ordered) basis $\cc_i$ of $C_i(M;\C^2_\rho)$ coming from $i$-cells.
We next choose a basis $\bb_i$ of $\Im\partial_{i+1}$ and its lift $\tilde{\bb}_i$ to $C_{i+1}(M;\C^2_\rho)$.
It follows from $H_\ast(M;\C^2_\rho)=0$ that the union $\bb_i\tilde{\bb}_{i-1}$ is a basis of $C_i(M;\C^2_\rho)$.
Let $[\bb_i\tilde{\bb}_{i-1}/\cc_i]$ denote the change of basis matrix from $\cc_i$ to $\bb_i\tilde{\bb}_{i-1}$.
Now, the \emph{$\SL(2,\C)$-Reidemeister torsion} (or simply \emph{Reidemeister torsion}) $\tau_\rho(M)$ of $M$ associated to $\rho$ is defined by
\[
\tau_\rho(M) = \prod_{i=0}^3 \det[\bb_i\tilde{\bb}_{i-1}/\cc_i]^{(-1)^{i+1}} \in \C^\times.
\]
See \cite{Joh88} and Section~1 of \cite{Kit94S,Kit94E} for details.
For instance, when $M$ is the complement of a 2-bridge knot $K$, the fundamental group $\pi_1(M)$ has a presentation with generators $x,y$ and a relator $r$.
Then, for an acyclic representation $\rho$, one can derive the formula
\[
\tau_\rho(M) = \frac{\det(\rho(\partial r/\partial y))}{\det(\rho(x)-I_2)},
\]
where $\partial r/\partial y$ denotes a Fox derivative of $r$.

We next consider the $\SL(2,\C)$-character variety $X(M)$, which is the GIT quotient of the $\SL(2,\C)$-representation variety $R(M)=\Hom(\pi_1(M),\SL(2,\C))$.
Let $X^\irr(M)$ denote the subspace of $X(M)$ consisting of conjugacy classes of irreducible representations.
In \cite[p.~54]{Joh88}, when $X^\irr(M)$ is a finite set, the \emph{torsion polynomial} $\sigma_M(t)$ of $M$ is defined by
\[
\sigma_M(t)=\prod_{[\rho] \in X^\irr(M),\ \text{acyclic}}(t-\tau_\rho(M)) \in \Q[t].
\]

\begin{remark}
\label{rem:convention}
Reidemeister torsion is sometimes defined to be $\prod_{i=0}^3 \det[\bb_i\tilde{\bb}_{i-1}/\cc_i]^{(-1)^{i}}$.
Our $\tau_\rho(M)$ is the same as Reidemeister torsion in \cite{Joh88, Kit16N, KiTr}, but the inverse of Reidemeister torsion in \cite{Kit94S, Kit94E}.
This difference is crucial for the results in this paper.
For instance, $\tau_\rho(S^3_1(4_1))$ is an algebraic integer, but its inverse is not since the constant term of $\sigma_{S^3_1(4_1)}(t)=t^3-12t^2+20t-8$ is not $\pm 1$.
\end{remark}

\subsection{The $A$-polynomial of knots}
Let $K$ be a knot in $S^3$ and recall that $E(K)$ denotes the complement of an open tubular neighborhood of $K$.
We consider the algebraic subset $U$ of $R(E(K))$ consisting of representations $\rho$ such that $\rho(\lambda)$ and $\rho(\mu)$ are upper triangular matrices, where $\lambda$ and $\mu$ denote a preferred longitude and a meridian, respectively.
Now we define a map $\xi\colon U \to \C^2$ by $\xi(\rho)=(\rho(\lambda)_{11}, \rho(\mu)_{11})$, where $X_{ij}$ denotes the $(i,j)$-entry of a matrix $X$.
The Zariski closure of the image of $\xi$ is some algebraic curves and points.
The product of the defining polynomials of the algebraic curves is known to be defined over $\Z$ so that the coefficients are relatively prime.
The resulting polynomial $A_K(L,M) \in \Z[L,M]$ is called the \emph{$A$-polynomial} of $K$, which is defined up to sign.
Following \cite{CCGLS94}, we drop the factor $L-1$ corresponding to abelian representations from the $A$-polynomial, that is, $A_\text{unknot}(L,M)=1$.

For 2-bridge knots, their $A$-polynomials are computed from Riley polynomials (see \cite[Section~7]{CCGLS94} and \cite[Section~3]{CoLo96}).
Here we recall the definition of the Riley polynomial introduced in \cite{Ril84}.
Let $K$ be a 2-bridge knot.
Then $\pi_1(E(K))$ has a presentation of the form $\ang{x,y \mid wx=yw}$, where $w$ is a word in $x$ and $y$.
For $s,t \in \C$ with $s\neq 0$, we consider the representation $\rho_{s,t}$ of the free group $\ang{x,y}$ by
\[
\rho_{s,t}(x) =
\begin{pmatrix}
 s & 1 \\
 0 & 1/s
\end{pmatrix},\quad
\rho_{s,t}(y) =
\begin{pmatrix}
 s & 0 \\
 -t & 1/s
\end{pmatrix}
\]
and define the \emph{Riley polynomial} $\phi(s,t)$ of $K$ by
\[
\phi(s,t)= \rho_{s,t}(w)_{11}+(s-s^{-1})\rho_{s,t}(w)_{12} \in \Z[s^{\pm 1},t].
\]
It is shown in \cite[Theorem~1]{Ril84} that $\rho_{s,t}$ factors through $\pi_1(E(K))$ if and only if $\phi(s,t)=0$ holds.
Moreover, every non-abelian representation is conjugate to some $\rho_{s,t}$.
Note that $\rho_{s,t}$ is irreducible if and only if $t\neq 0$.

The $A$-polynomial of a 2-bridge knot is obtained from its Riley polynomial by eliminating $t$.
More precisely, we consider the resultant
\[
\res_t(L-\rho_{M,t}(\lambda)_{11}, \phi(M,t)) \in \Z[L, M^{\pm 1}]
\]
of polynomials in $t$, and the $A$-polynomial is a factor of this resultant.
See Section~\ref{subsec:resultant} for the definition of the resultant.

\begin{remark}
The variable $M$ in $A_K(L,M)$ can be taken as the variable $s$ in $\phi(s,t)$ for a 2-bridge knot $K$. 
In this paper, however, we use a pair $(L,M)$ for $A_K(L,M)$ and $(s,t)$ for $\phi(s,t)$ under the conventions of these polynomials.
\end{remark}

\begin{example}
Let us consider the case $K=J(2,4)=5_2$.
Then $w=[y,x^{-1}]^2$, $\mu=x$, and $\lambda=[x,y^{-1}]^2[y,x^{-1}]^2$.
Thus, we have
\begin{align*}
 \phi(s,t) &= \left(2(s^2+s^{-2})-3\right) t^2+\left(-(s^4+s^{-4})+3 (s^2+s^{-2})-6\right) t+2(s^2+s^{-2})-t^3-3, \\
 A_{5_2}(L,M) &= -L^3 + L^2 (1 - 2 M^2 - 2 M^4 + M^8 - M^{10}) \\
 &\qquad + L M^4 (-1 + M^2 - 2 M^6 - 2 M^8 + M^{10}) - M^{14}.
\end{align*}
Let us give an observation.
Since $\phi(i,t)=-t^3-7t^2-14t-7$ and $A_{5_2}(L,i)=(L-1)^3$, there are three irreducible representations of $\pi_1(E(5_2))$ such that the restrictions to $\pi_1(\partial E(5_2))$ are identical.
However, one can check that $\tau_{\rho_{i,t}}(E(5_2))=(-t^3-3t^2+2t+9)/2$, and thus $\tau_{\rho_{i,t}}(E(5_2))$'s are distinct for the three representations.
That is, the Reidemeister torsion $\tau_\rho(E(5_2))$ is not necessarily determined by the restriction $r(\rho)=\rho|_{\partial E(5_2)}$.
On the other hand, $\tau_\rho(E(5_2))$ can only take finitely many values while $r(\rho)$ is fixed.
In particular, it does not vary continuously.
\end{example}

\subsection{Resultants of polynomials and algebraic integers}
\label{subsec:resultant}
Let $R$ be an integral domain.
Let $f(x)=a_0x^m+a_1x^{m-1}+\dots+a_m$ and $g(x)=b_0x^n+b_1x^{n-1}+\dots+b_n \in R[x]$ with $a_0,b_0 \neq 0$.
Then the \emph{resultant} of $f$ and $g$ is defined by
\[
\res_x(f,g) =
\begin{vmatrix}
 a_0 & a_1 & \cdots & a_{m-1} & a_m & & & \\
 & a_0 & a_1 & \cdots & a_{m-1} & a_m & & \\
 & & \ddots & \ddots & & \ddots & \ddots & \\
 & & & a_0 & a_1 & \cdots & a_{m-1} & a_m \\
 b_0 & b_1 & \cdots & b_{n-1} & b_n & & & \\
 & b_0 & b_1 & \cdots & b_{n-1} & b_n & & \\
 & & \ddots & \ddots & & \ddots & \ddots & \\
 & & & b_0 & b_1 & \cdots & b_{n-1} & b_n
\end{vmatrix}
\in R,
\]
which is the determinant of an $(n+m)\times(m+n)$ matrix.
It is well-known that $\res_x(f,g)=0$ if and only if $f$ and $g$ have a common zero in the algebraic closure of the quotient field of $R$.
The resultant is useful to eliminate variables from equations.
For instance, let $f(x,y)=2x^2 + y^2 - 1$, and $g(x,y)=x y - 1 \in R[x]$, where $R=\Z[y]$.
Then one has $\res_x(f,g)=y^4-y^2+2 \in \Z[y]$.
In particular, the $y$-coordinates of common zeros of $f$ and $g$ are algebraic integers.
We denote by $\A$ the ring of algebraic integers.
It is known that the ring $\A$ is integrally closed in $\C$.
See \cite[Sections~0.2 and 5.2]{MaRe03} for fundamental properties of algebraic integers and interesting relations to 3-manifolds.

Throughout this paper, we use some basic properties of the resultant.
For example, we use the facts that $\res_x(f,g)=a_0^n\prod_{g(\zeta)=0}f(\zeta)$ and $\res_x(f,gh)=\res_x(f,g)\res_x(f,h)$ (see \cite[Chapter~3, Section~1]{CLO05} for instance).

The rest of this section is devoted to showing the key lemma for the proof of Proposition~\ref{prop:twist_knot}.
Before starting the proof, we give an easy observation.
For $f(x)=a_0x^m+a_1x^{m-1}+\dots+a_m$, consider a one-row matrix
$A=\begin{pmatrix}
 a_0 & a_1 & \cdots & a_m
\end{pmatrix}$.
By adding $x$ times the $j$th entry to the $(j+1)$st entry for $j=1,2,\dots,m-1$ in this order, we obtain
$\begin{pmatrix}
 a_0 & a_0x+a_1 & \cdots & f(x)
\end{pmatrix}$.
If we apply this procedure $k+1$ times ($0\leq k<m$) to $A$, then the $i$th entry is equal to $\sum_{j=0}^{i-1}\binom{k+j}{j}a_{i-1-j}x^j \in R[x]$.
One can check it by the equality $\binom{k+j-1}{j-1}+\binom{k+j-1}{j} = \binom{k+j}{j}$ about binomial coefficients.
In particular, the $(m-k)$th entry is equal to $\frac{1}{k!}f^{(k)}(x)$, where $f^{(k)}$ denotes the $k$th derivative of $f$.
For instance, when $m=4$ and $k=2$, we have
\[
\begin{pmatrix}
 a_0 & 3 a_0 x+a_1 & 6 a_0 x^2+3 a_1 x+a_2 & \ast & \ast 
\end{pmatrix}.
\]
Its 3rd entry is actually equal to $\frac{1}{2!}f^{(2)}(x) = \frac{1}{2}(12a_0 x^2+6a_1 x+2a_2)$.

\begin{lemma}
\label{lem:factor_of_res}
Let $\zeta \in R$ and $f,g \in R[x]$.
If $g(\zeta)^{m-k}$ divides $\frac{1}{k!}f^{(k)}(\zeta)$ for all $0\leq k<m$, then $g(\zeta)^m$ divides $\res_x(f,g)$.
\end{lemma}

\begin{proof}
In the definition of $\res_x(f,g)$, we first add $\zeta$ times the $j$th column to the $(j+1)$st column for $j=1,2,\dots,m+n-1$ in this order.
We next add $-\zeta$ times the $i$th row to the $(i-1)$st row for $i=2,3,\dots,n,n+2,n+3,\dots,m+n$:
\begin{align*}
\res_x(f,g) &=
\begin{vmatrix}
 a_0 & a_0\zeta+a_1 & \cdots & & f(\zeta) & \cdots & & \zeta^{n-1}f(\zeta) \\
 & a_0 & & \cdots & & f(\zeta) & & \\
 & & \ddots & \ddots & & \ddots & \ddots & \\
 & & & a_0 & a_0\zeta+a_1 & \cdots & & f(\zeta) \\
 b_0 & b_0\zeta+b_1 & \cdots & & g(\zeta) & \cdots & & \zeta^{m-1}g(\zeta) \\
 & b_0 & & \cdots & & g(\zeta) & & \\
 & & \ddots & \ddots & & \ddots & \ddots & \\
 & & & b_0 & b_0\zeta+b_1 & \cdots & & g(\zeta)
\end{vmatrix}\\
&=
\begin{vmatrix}
 a_0 & a_1 & \cdots & a_{m-1} & a_m & & & \\
 & a_0 & a_1 & \cdots & a_{m-1} & a_m & & \\
 & & \ddots & \ddots & & \ddots & \ddots & \\
 & & & a_0 & a_0\zeta+a_1 & \cdots & & f(\zeta) \\
 b_0 & b_1 & \cdots & b_{n-1} & b_n & & & \\
 & b_0 & b_1 & \cdots & b_{n-1} & b_n & & \\
 & & \ddots & \ddots & & \ddots & \ddots & \\
 & & & b_0 & b_0\zeta+b_1 & \cdots & & g(\zeta)
\end{vmatrix}.
\end{align*}
Since $g(\zeta)^m \mid f(\zeta)$, one can divide the $(m+n)$th column by $g(\zeta)$.
We then sweep out the bottom row by 1 at the bottom left:
\begin{align*}
\res_x(f,g) &=
g(\zeta)
\begin{vmatrix}
 a_0 & a_1 & \cdots & a_{m-1} & a_m & & & \\
 & a_0 & a_1 & \cdots & a_{m-1} & a_m & & \\
 & & \ddots & \ddots & & \ddots & \ddots & \\
 & & & a_0 & a_0\zeta+a_1 & \cdots & & f(\zeta)g(\zeta)^{-1} \\
 b_0 & b_1 & \cdots & b_{n-1} & b_n & & & \\
 & b_0 & b_1 & \cdots & b_{n-1} & b_n & & \\
 & & \ddots & \ddots & & \ddots & \ddots & \\
 & & & b_0 & b_0\zeta+b_1 & \cdots & & 1
\end{vmatrix}\\
&=
g(\zeta)
\begin{vmatrix}
 a_0 & a_1 & \cdots & a_{m-1} & a_m & & & \\
 & \ddots & \ddots & & \ddots & \ddots & & \\
 & & a_0 & a_1 & \cdots & a_{m-1} & a_m & \\
 & & & a_0-\ast & a_0\zeta+a_1-\ast & \cdots & & f(\zeta)g(\zeta)^{-1} \\
 b_0 & b_1 & \cdots & b_{n-1} & b_n & & & \\
 & \ddots & \ddots & & \ddots & \ddots & & \\
 & & b_0 & b_1 & \cdots & b_{n-1} & b_n & \\
 & & & 0 & 0 & \cdots & 0 & 1
\end{vmatrix},
\end{align*}
where $\ast$'s are divisible by $g(\zeta)^{m-1}$.
We now reduce the size of the matrix by 1 and apply almost the same column and row operations as above:
\begin{align*}
\res_x(f,g) &=
g(\zeta)
\begin{vmatrix}
 a_0 & a_1 & \cdots & a_{m-1} & a_m & & \\
 & \ddots & \ddots & & \ddots & \ddots & \\
 & & a_0 & a_1 & \cdots & a_{m-1} & a_m \\
 & & & a_0-\ast & a_0\zeta+a_1-\ast & \cdots & \\
 b_0 & b_1 & \cdots & b_{n-1} & b_n & & \\
 & \ddots & \ddots & & \ddots & \ddots & \\
 & & b_0 & b_1 & \cdots & b_{n-1} & b_n \\
\end{vmatrix} \\
&=
g(\zeta)
\begin{vmatrix}
 a_0 & a_1 & \cdots & & a_{m-1} & a_m & & \\
 & \ddots & \ddots & & & \ddots & \ddots & \\
 & & a_0 & a_0\zeta+a_1 & \cdots & & & \frac{1}{0!}f(\zeta) \\
 & & & a_0-\ast & 2a_0\zeta+a_1-\ast & \cdots & & \frac{1}{1!}f^{(1)}(\zeta)-\ast \\
 b_0 & b_1 & \cdots & b_{n-1} & b_n & & & \\
 & \ddots & \ddots & & \ddots & \ddots & & \\
 & & b_0 & b_1 & \cdots & b_{n-1} & b_n & \\
 & & & b_0 & b_0\zeta+b_1 & \cdots & & g(\zeta)
\end{vmatrix}.
\end{align*}
Since $g(\zeta)^{m-1} \mid \frac{1}{1!}f^{(1)}(\zeta)-\ast$, one can divide the $(m+n-1)$st column by $g(\zeta)$.
By sweeping out the bottom row, $\res_x(f,g)$ is equal to
\begin{align*}
g(\zeta)^2
\begin{vmatrix}
 a_0 & a_1 & \cdots & & a_{m-1} & a_m & & \\
 & \ddots & \ddots & & & \ddots & \ddots & \\
 & & a_0-\ast & a_0\zeta+a_1-\ast & \cdots & & & f(\zeta)g(\zeta)^{-1} \\
 & & & a_0-\ast' & 2a_0\zeta+a_1-\ast' & \cdots & & (f^{(1)}(\zeta)-\ast)g(\zeta)^{-1} \\
 b_0 & b_1 & \cdots & b_{n-1} & b_n & & & \\
 & \ddots & \ddots & & \ddots & \ddots & & \\
 & & b_0 & b_1 & \cdots & b_{n-1} & b_n & \\
 & & & 0 & 0 & \cdots & 0 & 1
\end{vmatrix},
\end{align*}
where each $\ast'$ is divisible by $g(\zeta)^{m-2}$.
The assumption and the previous observation guarantee that one can continue this process until the size of the matrix is $n\times n$.
We finally obtain $\res_x(f,g) = g(\zeta)^m|X|$, where $X$ is some $n\times n$ matrix over $R$, and hence $g(\zeta)^m \mid \res_x(f,g)$.
\end{proof}

\section{Reidemeister torsion and algebraic integers}
\label{sec:algebraic_integer}
The aim of this section is to prove Theorems~\ref{thm:twist_knot} and \ref{thm:4_1}.

\subsection{The proof of Theorem~\ref{thm:twist_knot}}

\begin{lemma}
Let $K$ be a $2$-bridge knot.
Then the coefficient of the leading term of the $A$-polynomial $A_K(L,M)$ with respect to $L$ is a unit, namely a power of $M$ up to sign.
Moreover, it holds for the lowest degree term as well.
\end{lemma}

\begin{proof}
First recall that $\phi(s,t) \in \Z[s^{\pm 1},t]$ denotes the Riley polynomial of $K$.
Since $A_K(L,M)$ is a factor of the resultant
\[
f(L,M)=\res_t(\phi(M,t), L-\rho_{M,t}(\lambda)_{11}) \in \Z[L,M^{\pm 1}],
\]
it suffices to prove that the leading term of $f(L,M)$ with respect to $L$ is a unit.
By \cite[Lemma~2]{Ril84}, the coefficient of the leading term of $\phi(M,t)$ with respect to $t$ is $\pm 1$.
It follows from the definition of the resultant that the leading term of $f(L,M)$ with respect to $L$ is a power of $L$ up to sign.

Now, the latter follows from the equality $A_K(L,M)=A_K(L^{-1},M^{-1})$ up to units in $\Z[L,M]$ (see \cite[Proposition~4.2(1)]{CoLo96}).
\end{proof}

Before the next lemma, we recall a boundary slope of a knot $K$.
For an incompressible surface $S$ in $E(K)$, its boundary consists of parallel simple closed curves on $\partial E(K)$.
Thus, $S$ defines an element $p[\mu]+q[\lambda] \in H_1(\partial E(K);\Z)$ up to sign, and then $p/q \in \Q\cup\{1/0\}$ is called a \emph{boundary slope} of $K$.
It is proved in \cite[Theorem~3.4]{CCGLS94} that the slope of a side of the Newton polygon $N(A_K)$ of $A_K(L,M)$ is a slope of $K$.
Combining this theorem and \cite[Theorem~1(b)]{HaTh85}, we have that slopes of sides of $N(A_K)$ are even integers for a 2-bridge knot $K$.

\begin{lemma}
\label{lem:M_in_A}
Let $K$ be a $2$-bridge knot and let $p=\pm 1$ and $q>0$.
Then the coefficients of the highest and lowest degree terms of $\res_L(A_K(L,M), M^pL^q-1) \in \Z[M^{\pm 1}]$ are $\pm 1$.
\end{lemma}

\begin{proof}
It suffices to discuss only the highest degree term of the resultant since $A_K(L,M)=A_K(L^{-1},M^{-1})$ up to units in $\Z[L,M]$.
Let $d=\deg_L A_K(L,M)$ and suppose $p=-1$.
Let $c L^iM^j$ be the highest degree term of $A_K(L,M)$ with respect to the lexicographic order of $(\deg_M, \deg_L)$.
Since this monomial corresponds to a corner $v$ of the Newton polygon $N(A_K)$, \cite[Theorem~11.3]{CoLo96} implies $c=\pm 1$.
When we choose $c M^j$'s in the first $|q|$ rows of the matrix appearing in the resultant, the highest degree term is $\pm(c M^j)^{q}(M^{-1})^{d-i}1^i = \pm M^{jq-d+i}$.
Now, let $l$ be the line through $v$ with slope $p=-1$.
Since the slopes of the sides of $N(A_K)$ are even, $N(A_K)$ lies in the lower left half space.
Therefore, all terms in the resultant except $\pm M^{jq-d+i}$ have a lower degree.

In the case $p=1$, we can show the statement in a similar way by considering the lexicographic order of $(\deg_M, -\deg_L)$.
\end{proof}

As a consequence of Lemma~\ref{lem:M_in_A}, if $\rho_{s_0,t_0}$ is a representation satisfying the condition of Theorem~\ref{thm:twist_knot}, then $s_0$ and $s_0^{-1}$ are algebraic integers.

\begin{lemma}
\label{lem:s0t0}
For a root $(s_0,t_0)$ of $\phi(s,t)$, if $s_0, s_0^{-1} \in \A$, then $t_0 \in \A$.
\end{lemma}

\begin{proof}
By \cite[Lemma~2]{Ril84}, $\phi(s_0,t) \in \A[t]$ is monic.
Since the ring $\A$ of algebraic integers is integrally closed in $\C$, we conclude $t_0 \in \A$.
\end{proof}

In the rest of this section, let $K=J(2,2m)$ ($m\neq 0$).
Define $d(m)=\deg_t\phi(1,t)$, then by \cite[Lemma~2]{Ril84}, we have
\[
d(m)=
\deg_t\phi(s,t)=
\begin{cases}
 2m-1 & \text{if $m>0$,} \\
 -2m & \text{if $m<0$.}
\end{cases}
\]

\begin{lemma}
\label{lem:diff_of_A}
 $\frac{1}{n!}\frac{\partial^nA_K}{\partial L^n}(-1,M)$ is divisible by $(M^2-1)^{d(m)-n}$ for $0\leq n< d(m)$.
\end{lemma}

\begin{proof}
The proof is by induction on $m$.
The cases $m=-1,0,1,2$ are directly confirmed by \cite[Theorem~1]{HoSh04}.
For instance, in the case $m=2$, we have
\begin{align*}
 A_K(-1,M) &= (M^2-1)^3 \left(2 M^8+4 M^6+5 M^4+4 M^2+2\right), \\
 \frac{1}{1!}\frac{\partial A_K}{\partial L}(-1,M) &= -(M^2-1)^2 \left(M^{10}-M^6-4 M^4-6 M^2-5\right), \\
 \frac{1}{2!}\frac{\partial^2A_K}{\partial L^2}(-1,M) &= (M^2-1) \left(M^8+2 M^2+4\right).
\end{align*}
Suppose $m>2$.
Following \cite{HoSh04}, let
\begin{align*}
 x(L,M)&= -L+L^2 +2LM^2 +M^4 +2LM^4 +L^2M^4 +2LM^6 +M^8 -LM^8, \\
 y(L,M)&= M^4(L+M^2)^4.
\end{align*}
Then one can check that $(M^2-1)^{2-k} \mid \frac{1}{k!}\frac{\partial^k x}{\partial L^k}(-1,M)$ for $0\leq k\leq 2$ and $(M^2-1)^{4-k} \mid\frac{1}{k!} \frac{\partial^k y}{\partial L^k}(-1,M)$ for $0\leq k\leq 4$.
By the recursive relation in \cite[Theorem~1]{HoSh04} and Leibniz's rule, we have
\[
\frac{1}{n!}\frac{\partial^nA_K}{\partial L^n} = \sum_{k=0}^{2} \frac{1}{k!}\frac{\partial^k x}{\partial L^k} \frac{1}{(n-k)!}\frac{\partial^{n-k}A_{J(2,2m-2)}}{\partial L^{n-k}} - \sum_{k=0}^{4} \frac{1}{k!}\frac{\partial^k y}{\partial L^k} \frac{1}{(n-k)!}\frac{\partial^{n-k}A_{J(2,2m-4)}}{\partial L^{n-k}}.
\]
When $L=-1$, by the induction hypothesis, each term in the sums is divisible by $(M^2-1)^{2m-1-n}$.

The cases $m<-1$ also follow from the recursive relation in a similar way.
\end{proof}

\begin{proposition}
\label{prop:twist_knot}
Let $p=\pm 1$.
Then, $\res_L(A_K(L,M), M^pL^q-1)$ is divisible by $(M^p(-1)^{q}-1)^{d(m)}$.
\end{proposition}

\begin{proof}
Note that $M^p(-1)^{q}-1 \mid M^2-1$ when $p=\pm 1$.
Now, Lemma~\ref{lem:diff_of_A} allows us to apply Lemma~\ref{lem:factor_of_res} to $\zeta= -1 \in \C[M]$ and
\[
f(L,M)=A_K(L,M),\ g(L,M)=ML^q-1 \in \C[M][L]=\C[L,M].
\]
This completes the proof.
\end{proof}

Here recall that $\tau_\rho(E(K))$ is written of the form
\[
\frac{\det(\rho_{s_0,t_0}(\partial r/\partial y))}{\det(\rho_{s_0,t_0}(x)-I_2)} = -s_0(s_0-1)^{-2}F(s_0,t_0),
\]
where $F(s,t) \in \Z[s^{\pm 1},t]$.
In the situation of Theorem~\ref{thm:twist_knot}, $s_0$ and $F(s_0,t_0)$ are algebraic integers by Lemmas~\ref{lem:M_in_A} and \ref{lem:s0t0}.
Hence, it suffices for proving the theorem to consider whether $(s_0-1)^{-1}$ is an algebraic integer.

\begin{lemma}
\label{lem:alpha-1}
Let $f(x) \in \Z[x]$ be a monic polynomial and let $\alpha \in \C\setminus\{1\}$ be a root of $f$.
If $f(1)=\pm 1$, then $(\alpha-1)^{-1}$ is an algebraic integer.
\end{lemma}
\begin{proof}
Let $\beta=(\alpha-1)^{-1}$, namely $\alpha=\beta^{-1}+1$.
Since $f$ is monic, we have
\[
0=f(\alpha)=f(\beta^{-1}+1)=\beta^{-\deg f}+\cdots+f(1).
\]
Here, by $f(1)=\pm 1$, there exists a monic polynomial $g \in \Z[x]$ such that $g(\beta)=0$.
\end{proof}

\begin{proof}[Proof of Theorem~\ref{thm:twist_knot}]
Recall that $q$ is odd.
Let us show that $(s_0-1)^{-1} \in \A$.
By Lemma~\ref{lem:alpha-1}, it suffices to find a monic polynomial $f(s)$ such that $f(s_0)=0$ and $f(1)=\pm 1$.
Proposition~\ref{prop:twist_knot} implies that
\[
f(s):= (s+1)^{-d(m)} \res_L(A_K(L,s), sL^q-1)
\]
lies in $\Z[s^{\pm 1}]$ and its is monic.
By \cite[Corollary~3(i)]{HoSh04}, we have
\begin{align*}
 f(1) &= (1+1)^{-d(m)} \res_L(\pm(L+1)^{d(m)}, L^q-1) \\
 &= \pm 2^{-d(m)} \res_L(L+1, L^q-1)^{d(m)} = \pm 1.
\end{align*}
Therefore, $(s_0-1)^{-1}$ is an algebraic integer.
\end{proof}

\begin{remark}
When $q$ is even, $(s_0-1)^{-1}$ is not necessarily an algebraic integer.
For example, in the case where $J(2,-2)=4_1$ and $(p,q)=(1,2)$, $s_0$'s are the roots of the polynomial
\[
f(x)=x^{14}+2 x^{13}+x^{12}-4 x^{10}-8 x^9-10 x^8-13 x^7-10 x^6-8 x^5-4 x^4+x^2+2 x+1.
\]
Here, $f(x)$ is irreducible since it is the product of two irreducible polynomials of degree 7 over the prime field $\mathbb{F}_2$ of order $2$ and is the product of two irreducible polynomials of degree 4 and 10 over $\mathbb{F}_3$.
Thus, if $(s_0-1)^{-1} \in \A$, then $f(1)=\pm 1$.
However, we have $f(1)=-49$.

On the other hand, $\tau_\rho(E(4_1)) = -2(s_0+s_0^{-1}-1)$ is an algebraic integer since $s_0^{\pm 1} \in \A$.
The authors predict that Theorem~\ref{thm:twist_knot} holds for every integer $q$.
\end{remark}

\subsection{The proof of Theorem~\ref{thm:4_1}}
This subsection is devoted to proving that $\tau_\rho(S^3_{p/q}(4_1))$ is an algebraic integer when $p=1$ or $q=1$.
We also observe that it is an algebraic integer when $p/q=2/3$ in Example~\ref{ex:2/3}.
Moreover, $\tau_\rho(S^3_{1/2}(5_2))$ is an algebraic integer as well (see Example~\ref{ex:5_2} and Problem~\ref{prob:tau_closed}).

Let $q\neq 0$ and let $\rho$ be an acyclic representation of $\pi_1(S^3_{p/q}(4_1))$.
Then
\[
\tau_\rho(S^3_{p/q}(4_1)) = \frac{-2(s+s^{-1}-1)}{\det(\rho(\mu^{p'}\lambda^{q'})-I_2)} = 2(s+s^{-1}-1)s^{p'}L^{q'}(s^{p'}L^{q'}-1)^{-2},
\]
where $pq'-qp'=1$ and $L=\rho(\lambda)_{11}$.
We divide the proof of Theorem~\ref{thm:4_1} into two cases: (I) $p=1$, $q\neq 0$ and (II) $p\neq 0$, $q=1$.
Recall that
\begin{align*}
A_{4_1}(L,M) &= L^2M^4+L(-M^8-1+M^6+M^2+2M^4)+M^4 \\
 &= \left(L^2+L(-(M+M^{-1})^4+5(M+M^{-1})^2-2)+1\right)M^4.
\end{align*}

\begin{proof}[Proof of Theorem~\ref{thm:4_1}(I)]
Let $p=1$, $q\neq 0$.
Then one can take $(p',q')=(0,1)$.
By Lemmas~\ref{lem:M_in_A} and \ref{lem:s0t0}, it suffices to show that $(L_0-1)^{-1}$ is an algebraic integer.
First, by direct computations, we have $\tr\rho(\mu)\neq -2$.
Indeed, if $s=-1$, then
\[
\rho(\lambda)=
\begin{pmatrix}
 -1 & \pm 2 i \sqrt{3} \\
 0 & -1 \\
\end{pmatrix},
\]
and hence $\rho(x)\rho(\lambda)^q \neq I_2$.
See also \cite[Lemma~6.2]{KiTr}.

Let $f(L)=A_{4_1}(L,L^{-q})$ and $g(L)=L^{-q}+L^q \in \Z[L^{\pm 1}]$.
Then, we have $f(-1)=0$ and
\begin{align*}
f'(L) &= \left(2L+(-g(L)^4+5g(L)^2-2)+L(-4g'(L)g(L)^3+10g'(L)g(L))\right)L^{-4q} \\
 &\qquad -4q\left(L^2+L(-g(L)^4+5g(L)^2-2)+1\right)L^{-4q-1}.
\end{align*}
It follows from $g'(-1)=0$ that $f'(-1)=0$, and hence $f(L)=(L+1)^2h(L)$ for some $h(L) \in \Z[L^{\pm 1}]$.
Here, $h(L_0)=0$ and $h(1)=1$ since $L_0\neq -1$ and $f(1)=4$.
Therefore, Lemma~\ref{lem:alpha-1} implies $(L_0-1)^{-1} \in \A$.
\end{proof}

Next, we discuss the case (II), that is, $q=1$.
Then one can take $(p',q')=(-1,0)$.
If $(s-1)^{-1}$ is an algebraic integer, we can obtain the desired result.
However, the case $p=4$, we have $A_{4_1}(s^{-4},s)=s^{-2}(s^2+1)^2$, and hence $s=\pm i$.
The minimal polynomial of $(\pm i-1)^{-1}$ is $2x^2+2x+1$, and thus they are not algebraic integers.
This observation suggests that one needs to show that $\tau_\rho(S^3_{p}(4_1)) = 2(s^2-s+1)/(s-1)^2$ is an algebraic integer directly.
Note that
\[
A_{4_1}(s^{-p},s) = s^{2-p}+2 s^{4-p}+s^{6-p}-s^{8-p}-s^{-p}+s^{4-2 p}+s^4.
\]

\begin{lemma}
\label{lem:div_by_16}
The coefficient of the leading term of
\[
\res_s(\tau(s-1)^2-2(s^2-s+1), A_{4_1}(s^{-p},s)) \in \Z[\tau]
\]
is equal to $\pm 16$.
Moreover, it is divisible by $4(2\tau-3)^2$ if $p$ is odd and by $16$ if $p$ is even.
\end{lemma}

\begin{proof}
Let $n=\deg A_{4_1}(s^{-p},s)$ and
\begin{align*}
h(\tau) &= \res_s(\tau(s-1)^2-2(s^2-s+1), A_{4_1}(s^{-p},s)) \\
 &=
\begin{vmatrix}
 \tau-2 & -2\tau+2 & \tau-2 & & \\
 & \ddots & & \ddots & \\
 & & \tau-2 & -2\tau+2 & \tau-2 \\
 \ast & \cdots & & \ast & \\
 & \ast & \cdots & & \ast 
\end{vmatrix}.
\end{align*}
Then, by the definition of the resultant,
\begin{align*}
 \frac{1}{n!}\frac{d^n}{d\tau^n}h(\tau) &= \res_s((s-1)^2, A_{4_1}(s^{-p},s)) \\
 &= \pm A_{4_1}(1,1)^2 = \pm 16.
\end{align*}
When $p$ is odd, we first note that $(s+1)^2 \mid A_{4_1}(s^{-p},s)$.
Thus, $h(\tau)$ is a multiple of $\res_s\left(\tau(s-1)^2-2(s^2-s+1), (s+1)^2\right) = \pm(4\tau-6)^2$.

We next consider the case $p$ even.
For the matrix in the resultant $h(\tau)$, we add the $j$th column to the $(j+1)$st column for $j=1,2,\dots,n+1$ in this order twice.
Then one obtains
\[
\begin{vmatrix}
 \tau-2 & -2 & -4 & \cdots & & -2n-2 \\
 & \tau-2 & -2 & -4 & \cdots & -2n \\
 & & \ddots & & & \vdots \\
 & & & \tau-2 & -2 & -4 \\
 \ast & \cdots & & \ast & \ast' & \ast' \\
 & \ast & \cdots & \ast & \ast' & \ast'
\end{vmatrix},
\]
where four $\ast'$'s are multiples of $4$.
Since $\tau$'s appear only in the diagonal and the entries in the upper triangle are even, the coefficient of $\tau^k$ is a multiple of $16$ unless $k=n-1, n-2, n-3$.
Here note that there are even number of odd integers in the $(n+i)$th row for $i=1,2$ since $2 \mid A_{4_1}(1,1)$.
Thus, the coefficient of $\tau^k$ is a multiple of $16$ when $k=n-1, n-2$.

Now, the coefficients in $h(\tau)$ are multiples of $16$ except the coefficient of $\tau^{n-3}$.
Hence, it suffices to see that $16 \mid h(1)=\res_s(-s^2-1,A_{4_1}(s^{-p},s))$.
By a property of the resultant, we have $h(1) = \pm A_{4_1}(i^{-p},i)A_{4_1}((-i)^{-p},-i)$, which is equal to $0$ or $\pm 16$.
\end{proof}

\begin{proof}[Proof of Theorem~\ref{thm:4_1}(II)]
By Lemma~\ref{lem:div_by_16}, when $p$ is odd, $4^{-1}(2\tau-3)^{-2}h(\tau)$ is in $\Z[\tau]$ and monic.
When $p$ is even, $16^{-1}h(\tau)$ is in $\Z[\tau]$ and monic as well.
Hence, we conclude that $\tau_\rho(S^3_{p}(K))$ is an algebraic integer.
\end{proof}

We now pose the following problem.

\begin{problem}
\label{prob:tau_closed}
It is natural to ask whether $\tau_\rho(S^3_{p/q}(4_1))$ is always an algebraic integer.
More generally, whether $\tau_\rho(S^3_{p/q}(K))$ is an algebraic integer for any knot $K$ with $\dim_\C X^\irr(E(K))=1$.
\end{problem}


We confirm that $\tau_\rho(S^3_{p/q}(K))$ is an algebraic integer for every irreducible representation $\rho$ in the cases of $S^3_{2/3}(4_1)$ and $S^3_{1/2}(5_2)$.

\begin{example}
\label{ex:2/3}
Let us consider the case $p/q=2/3$.
Then there are 12 representations of $\pi_1(S^3_{2/3}(4_1))$ as follows.
\[
\begin{array}{l | l | l}
 s & t & \tau_\rho(S^3_{2/3}(4_1)) \\ \hline
 -0.200325+0.979729 i & -3.42754 & 0.738094 \\
 0.200325+0.979729 i & -3.42754 & 5.85638 \\
 -0.490393+0.871501 i & -2.21504 & 2.38654 \\
 -1.30664+0.0498758 i & -0.392004+0.724199 i & 5.8872-0.943648 i \\
 -0.264802+0.964303 i & -1.43838 & 8.4028 \\
 0.490393+0.871501 i & -2.21504 & 0.0164217 \\
 0.264802+0.964303 i & -1.43838 & 0.258749 \\
 1.30664+0.0498758 i & -0.392004-0.724199 i & -0.717017+0.0236658 i \\
 -0.764207+0.0291705 i & -0.392004-0.724199 i & 5.8872+0.943648 i \\
 0.764207+0.0291705 i & -0.392004+0.724199 i & -0.717017-0.0236658 i \\
 -0.615146 & -0.135036 & 1.60981 \\
 0.615146 & -0.135036 & 94.3908 \\
\end{array}
\]
Using the resultant, one finds a polynomial $f(x)^2g(x)$ which is zero at these exact values of the Reidemeister torsion, where 
\begin{align*}
f(x)&= x^{12}-124 x^{11}+3142 x^{10}-34792 x^9+196796 x^8-561760 x^7+627280 x^6 \\
& \quad +254848 x^5-866240 x^4+153088 x^3+253696 x^2-66560 x+1024
\end{align*}
and $g(x)$ is a certain polynomial.
We can confirm that $g(x)\neq 0$ for the above 12 values $\tau_\rho(S^3_{2/3}(4_1))$.
Thus, they must be the zeros of $f$, and hence they are algebraic integers.

Since $f(94)<0<f(95)$, the zero of $f$ approximated by $94.3908$ must be a real number.
This value is larger than the absolute values of the other 11 zeros.
Hence, this zero is a Perron number.
Here the corresponding representation is an $\SL(2,\R)$-representation.
Indeed, $s$'s and $t$'s are respectively zeros of the polynomials
\begin{align*}
 & s^{24}-3 s^{22}-3 s^{20}+8 s^{18}+12 s^{16}-7 s^{14}-20 s^{12}-7 s^{10}+12 s^8+8 s^6-3 s^4-3 s^2+1, \\
 & t^6+8 t^5+23 t^4+31 t^3+23 t^2+10 t+1.
\end{align*}
They have zeros between $0.6$ and $0.7$ and between $-0.2$ and $-0.1$, respectively.
Therefore, one of the $\SL(2,\R)$-representations gives a Perron number.
Such a phenomenon is numerically seen in \cite{Kit16N} as well.
Also, the Reidemeister torsions of the Brieskorn homology 3-sphere $\Sigma(p,q,r)$ are real numbers (see Section~\ref{sec:Seifert}).
Thus, the maximal value $(4\sin\frac{\pi}{2p}\sin\frac{\pi}{2q}\sin\frac{\pi}{2r})^{-2}$ is a Perron number.
Note that $\pi_1(\Sigma(2,3,5))$ does not admit any $\SL(2,\R)$-representation and the maximal Reidemeister torsion comes from an $\SU(2)$-representation (see \cite[Remark~3.4]{KiYa16}).
\end{example}

\begin{example}
\label{ex:5_2}
There are 17 representations of $\pi_1(S^3_{1/2}(5_2))$ and three of them are not acyclic:
\[
\begin{array}{l | l | l}
 s & t & \tau_\rho(S^3_{1/2}(5_2)) \\ \hline
 -0.471842+0.881683 i & -2.87048 & 2.81243 \\
 0.165381+0.98623 i & -3.68825 & 12.3508 \\
 -0.200082+0.979779 i & -2.2146 & 0.490587 \\
 0.0681942+0.997672 i & -2.41933 & 0.313753 \\
 1.29286+1.35876 i & -0.182462-0.461334 i & 9.57556-0.520417 i \\
 -1.26355+0.363134 i & 0.158863+1.07159 i & 3.6856-0.147423 i \\
 -0.348139+0.937443 i & -1.1692 & 4.5522 \\
 1 & 0.21508-1.30714 i & 0 \\
 1 & 0.21508+1.30714 i & 0 \\
 1 & 0.56984 & 0 \\
 -0.859034+0.511919 i & -0.62449 & 7.42456 \\
 0.313791+0.949492 i & -1.80681 & 0.0670363 \\
 0.942666+0.333737 i & 0.0622582 & 148.658 \\
 0.709501+0.704704 i & -1.66753 & 1.06479 \\
 -0.731039+0.210094 i & 0.158863-1.07159 i & 3.6856+0.147423 i \\
 0.367529+0.386263 i & -0.182462+0.461334 i & 9.57556+0.520417 i \\
 -0.986232+0.16537 i & 0.445642 & 5.74328 \\
\end{array}
\]
In the same manner as Example~\ref{ex:2/3}, we can show that 14 values of $\tau_\rho(S^3_{1/2}(5_2))$ are the zeros of the monic polynomial
\begin{align*}
& x^{14}-210 x^{13}+10760 x^{12}-269160 x^{11}+3993232 x^{10} \\
& -38203808 x^9+245006784 x^8 -1067441024 x^7+3141232640 x^6-6091473408 x^5 \\
& +7422475264 x^4-5260713984 x^3+1942106112 x^2-314212352 x+13778944.
\end{align*}
Hence they are algebraic integers.

As similar to the previous example, the zero of the above polynomial approximated by $148.658$ is a Perron number.
The corresponding representation is numerically an $\SL(2,\R)$-representation by considering the conjugate $P^{-1}\rho(\textendash)P$, where
$P\approx
\begin{pmatrix}
 -0.950489 i & -1.97479 \\
 0.341848 & 0.341848 i \\
\end{pmatrix}
$.
\end{example}

\section{Reidemeister torsion of Seifert fibered spaces}
\label{sec:Seifert}
In this section, we show that the $\SL(2,\C)$-Reidemeister torsion $\tau_\rho(M)$ is an algebraic integer for a Seifert fibered space $M$ under a mild condition, which any Brieskorn homology 3-sphere satisfies.

\subsection{Chebyshev polynomials and variants}
We use Chebyshev polynomials and variants to prove the algebraic integrality. 
The next lemma follows from the fact that $\cos\frac{2k-1}{2n}\pi = \sin\frac{n-2k+1}{2n}\pi$ ($k=1,2,\dots,n$) are the roots of the Chebyshev polynomial of the first kind $T_n(x)$ defined by $T_n(\cos\theta)=\cos n\theta$.
Note that $T_{2n}(x)$ has the form of $2^{2n-1}x^{2n}+\dots+(-1)^n \in \Z[x]$.

\begin{lemma}
\label{lem:sin}
Let $a$ be a positive even integer.
Then $(\sin\frac{k}{2a}\pi)^{-1}$ is an algebraic integer for $k=-a+1,-a+3,\dots,a-1$.
\end{lemma}

\begin{proof}
As mentioned above, $\sin\frac{k}{2a}\pi$ gives a zero of $T_a(x)$. 
Then $(\sin\frac{k}{2a}\pi)^{-1}$ is a zero of $x^{a}T_a(1/x)=\pm x^{a}+\dots+2^{a-1}$. 
\end{proof}

Here let us introduce a variant of the normalized Chebyshev polynomial of the second kind.

\begin{definition}
Define a polynomial $V_n(x) \in \Z[x]$ inductively by $V_0(x)=1$, $V_1(x)=x-1$, and $V_n(x)=xV_{n-1}(x)-V_{n-2}(x)$ for $n\geq 2$.
\end{definition}

By definition, $V_n(x)$ is of the form $x^n+\dots\pm 1$.

\begin{lemma}
\label{lem:cos}
Suppose $\theta\neq (2k-1)\pi$ for any $k \in\Z$.
Then $V_n(2\cos\theta)=\cos(n+\frac{1}{2})\theta/\cos\frac{\theta}{2}$.
\end{lemma}

\begin{proof}
Use induction on $n$.
The cases $n=0,1$ follow from the equalities $\cos\frac{3}{2}\theta = \cos\theta \cos\frac{1}{2}\theta -\sin\theta \sin\frac{1}{2}\theta$ and $\sin\theta = 2\sin\frac{\theta}{2}\cos\frac{\theta}{2}$.
Suppose $n\geq 2$.
Then we have
\begin{align*}
 V_n(2\cos\theta)\cos\frac{\theta}{2} &= (2\cos\theta V_{n-1}(2\cos\theta)-V_{n-2}(2\cos\theta))\cos\frac{\theta}{2} \\
 &= 2\cos\theta \cos\left(n-\frac{1}{2}\right)\theta -\cos\left(n-\frac{3}{2}\right)\theta \\
 &= \cos\left(n-\frac{1}{2}\right)\theta \cos\theta -\sin\left(n-\frac{1}{2}\right)\theta \sin\theta \\
 &= \cos\left(n+\frac{1}{2}\right)\theta,
\end{align*}
where the second equality follows from the induction hypothesis.
\end{proof}

Therefore, the roots of $V_n(x)$ are $2\cos\frac{2k-1}{2n+1}\pi = 2\sin\frac{2n+3-4k}{2(2n+1)}\pi$ ($k=1,2,\dots,n$).
Since $V_n(0)=\pm 1$, we obtain the next consequence.

\begin{corollary}
\label{cor:2sin}
Let $a$ be a positive odd integer.
Then $(2\sin\frac{a-2k}{2a}\pi)^{\pm 1}$ is an algebraic integer for $k=-a+4,-a+8,\dots,a-2$.
\end{corollary}

\subsection{Seifert fibered spaces}

Let $M$ be an orientable Seifert fibered space with Seifert index  
\[
\{b,(\epsilon=o, g), (a_1,b_1),\dots,(a_m,b_m)\}, 
\]
whose base orbifold is a closed oriented surface with $m$ singular points. 
Here $b$ is an integer, $g$ is a non-negative integer and $a_i$, $b_i$ are coprime integers.

Now any non-trivial value of the Reidemeister torsion $\tau_\rho(M)$ for an irreducible $\SL(2,\C)$-representation $\rho\colon \pi_1(M) \to \SL(2,\C)$ is given as follows.
See \cite[Main Theorem]{Kit94S} and Remark~\ref{rem:convention}.

\begin{proposition}
\label{lem:tau_Seifert}
\[
\tau_\rho(M)^{-1} =
2^{4-m-g}\prod_{i=1}^m
\left(1-(-1)^{s_i}\cos\frac{r_i k_i\pi}{a_i}\right),
\]
where 
\begin{enumerate}
\item
$r_i,s_i\in\Z$ such that $a_is_i-b_ir_i=-1$ $(i=1,\dots,m)$,
\item
$k_i\in\Z$ such that $0\leq k_i\leq a_i$ and $k_i\equiv b_i\bmod 2$ $(i=1,\dots,m)$.
\end{enumerate}
\end{proposition}

\begin{remark}
A pair $(r_i, s_i )$ is not unique, 
but it depends on only $(a_i,b_i)$.
On the other hand $k_i$  does on a representation $\rho$. 
\end{remark}

We can see the following by Lemmas~\ref{lem:sin} and \ref{lem:cos}.
We write $m_o$ for the number of odd $a_i$'s.

\begin{proposition}
\label{prop:Seifert}
For a Seifert fibered space $M$ with an irreducible representation $\rho$,  
if $2m_o+g\geq 4$, then 
the Reidemeister torsion $\tau_\rho(M)$ is an algebraic integer.
\end{proposition}

\begin{proof}
Here we deform each factor, 
which contains a cosine function, of the formula in Lemma~\ref{lem:tau_Seifert} to a more simple form. 

\begin{itemize}
\item Case (I): $a_i$ is odd and $b_i$ is even. 

In this case we see $ s_i $ must be odd since $a_i s_i -b_ir_i=-1$.
Now,
\[
\begin{split}
1-(-1)^{ s_i }\cos\frac{r_i k_i\pi}{a_i}
&=1+\cos\frac{r_i k_i\pi}{a_i}\\
&=2\cos^2\left(\frac{r_i k_i\pi}{2a_i}\right)\\
&=2\sin^2\left(\frac{r_{i} {k}_i\pi}{2a_i}+\frac{\pi}{2}\right)\\
&=2\sin^2\left(\frac{2a_i\pi}{2a_i}-\frac{(r_{i}{k}_i+a_i)\pi}{2a_i}\right)\\
&=2\sin^2\left(\frac{a_i-r_i k_i}{2a_i}\pi\right).\\
\end{split}
\]
\item Case (II): $a_i$ is odd and $b_i$ is odd. 

In this case $ s_i $ may be odd and $r_i$ even by changing $ s_i $ to $ s_i +b_i$ and $r_i$ to $r_i+a_i$.
Similarly, we have
\[
1-(-1)^{ s_i }\cos\frac{r_i k_i\pi}{a_i}
=2\sin^2\left(\frac{a_i-r_i k_i}{2a_i}\pi\right).
\]

\item Case (III): $a_i$ is even. 

In this case $b_i$ must be odd and $ s_i $ may be odd by changing $ s_i $ to $ s_i +b_i$. 
Similarly to Cases (I) and (II), 
we see 
\[
1-(-1)^{ s_i }\cos\frac{r_i k_i\pi}{a_i}
=2\sin^2\left(\frac{a_i-r_i k_i}{2a_i}\pi\right).
\]
\end{itemize}

Now we see 
\[
\tau_\rho(M)
=
2^{2m_o+g-4}\prod_{a_i:\,\text{odd}}
\left(2\sin\frac{a_i-r_i k_i}{2a_i}\pi\right)^{-2}
\cdot
\prod_{a_i:\,\text{even}}
\left(\sin\frac{a_i-r_i k_i}{2a_i}\pi\right)^{-2} .
\]
Recall that $m_o$ is the number of odd $a_i$'s.
By Lemma~\ref{lem:sin} and Corollary~\ref{cor:2sin}, it can be seen that all factors
\[
2^{2m_o+g-4},\ 
\left(2\sin\frac{a_i-r_i k_i}{2a_i}\pi\right)^{-2}\ (\text{$a_i$: odd}),\ 
\left(\sin\frac{a_i-r_i k_i}{2a_i}\pi\right)^{-2}\ (\text{$a_i$: even}), 
\] 
are algebraic integers.
This completes the proof.
\end{proof}

Consider a Brieskorn homology 3-sphere $\Sigma(a_1,a_2,a_3)$. 
This is a Seifert fibered space with $m=3,~g=0$ and $m_o\geq 2$. 
Since there exist finitely many conjugacy classes of irreducible representations, 
the torsion polynomial can be defined. 

\begin{corollary}
The torsion polynomial $\sigma_{\Sigma(a_1,a_2,a_3)}(t) \in \Q[t]$ lies in $\Z[t]$.
\end{corollary}

Recall that we write $T(p,q)$ for the $(p,q)$-torus knot.
By similar arguments and Johnson's computation~\cite{Joh88}, we see the following. 

\begin{proposition}
\label{prop:torus_knot}
For any irreducible representation $\rho$, the Reidemeister torsion $\tau_\rho(E(T_{p,q}))$ is an algebraic integer.
\end{proposition}

\section{Continuous variation of Reidemeister torsion}
\label{sec:Looper-Long}
In this final section, we prove Theorem~\ref{thm:vary_conti} which asserts that for a certain knot $K_0$ the Reidemeister torsion $\tau_\rho(E(K_0))$ can vary continuously while the restriction $r(\rho)$ is fixed.
We construct the knot $K_0$ in Figure~\ref{fig:K0} following \cite[Section~8]{CoLo96}.

\begin{figure}[h]
 \centering
 \includegraphics[width=0.5\textwidth]{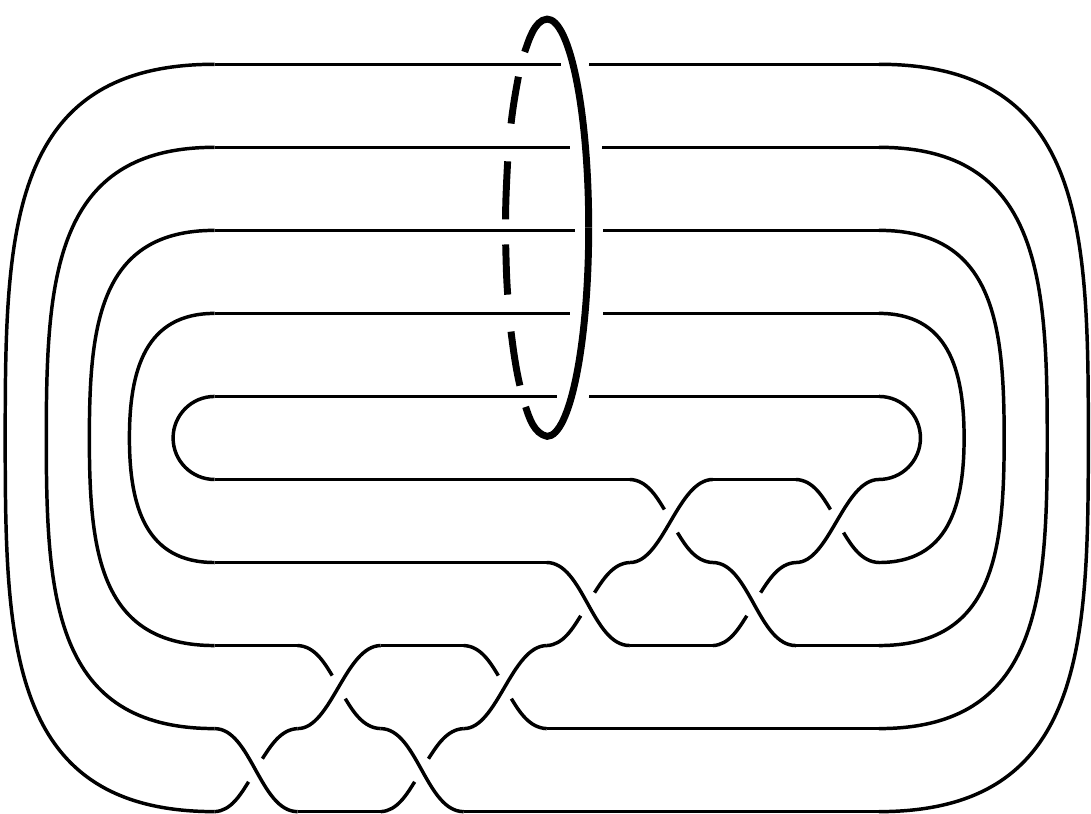}
 \caption{The knot $4_1\sharp 4_1$ with a braid axis.}
 \label{fig:4141}
\end{figure}

\begin{remark}
Since Figure~\ref{fig:K0} is an alternating diagram, $K_0$ is an alternating knot.
Then, we confirm that $K_0$ is prime by \cite[Theorem~4.4]{Lic97}.
The Alexander polynomial $\Delta_{K_0}(t)$ of $K_0$ is computed as follows:
\[
t^{12}-11 t^{11}+55 t^{10}-169 t^9+358 t^8-551 t^7+635 t^6-551 t^5+358 t^4-169 t^3+55 t^2-11 t+1.
\]
In particular, $K_0$ is not a torus knot, and thus it is hyperbolic by \cite[Corollary~2]{Men84}.
Moreover, since $\Delta_{K_0}(t)$ is monic, $K_0$ is a fibered knot of genus $6$ by \cite[Theorem~1.2]{Mur63}.
\end{remark}

Let $\varpi \colon E(K_0) \to E(4_1\sharp 4_1)$ denote a double branched cover whose branch set is the braid axis.
The map $\varpi$ induces a homomorphism $\varpi_\ast\colon \pi_1(E(K_0)) \to \pi_1(E(4_1\sharp 4_1))$ and a regular map $\varpi^\ast\colon X(E(4_1\sharp 4_1)) \to X(E(K_0))$.

Recall that irreducible representations of
\[
\pi_1(E(4_1)) = \ang{x,y \mid [y,x]^{-1}x=y[y,x]^{-1}}
\]
can be given by
\[
\rho(x) =
\begin{pmatrix}
 s & 1 \\
 0 & s^{-1}
\end{pmatrix},\quad
\rho(y) =
\begin{pmatrix}
 s & 0 \\
 -t & s^{-1}
\end{pmatrix},
\]
up to conjugate, where $s,t \in \C^\times$ satisfies $\phi(s,t)=0$.
Suppose $s \neq \pm 1$.
Consider irreducible representations $\rho_{s,t}\ast (P_u^{-1}\rho_{s,t'}P_u)$ of $\pi_1(E(4_1\sharp 4_1))$, where
\[
P_u=
\begin{pmatrix}
 u & (u-u^{-1})/(s-s^{-1}) \\
 0 & u^{-1}
\end{pmatrix}
\in \SL(2,\C)
\]
is an element of the centralizer of $\rho_{s,t}(x)$.
Then we define an irreducible representation $\rho_{s,t,u}$ of $\pi_1(E(K_0))$ by $\rho_{s,t,u}= \varpi^\ast(\rho_{s,t}\ast (P_u^{-1}\rho_{s,t}P_u))$.

\begin{proof}[Proof of Theorem~\ref{thm:vary_conti}]
Define
\[
C=\{\rho_{s,t,u} \mid s,t,u \in \C^\times,\ s\neq\pm 1,\ \phi(s,t)=0\}.
\]
There is a Wirtinger representation of $\pi_1(E(K_0))$ with $16$ generators and $15$ relations.
Then, by Mathematica, we compute the Reidemeister torsion as follows:
\[
\tau_{\rho_{s,t,u}}(E(K_0)) = \frac{f_2(s,t)u^2+f_0(s,t)+f_{-2}(s,t)u^{-2}}{s^{70}(s-1)^{19}(s+1)^{18}},
\]
where $f_j(s,t) \in \Z[s,t]$ ($j=2,0,-2$) are certain complicated polynomials.
Moreover, $\res_t(f_j(s,t), \phi(s,t)) \in \Z[s]$ is the product of six factors $s$, $s\pm 1$, $s^2\pm s-1$, and $g_j(s)$ with some multiplicities, where $g_j(s) \in \Z[s]$ for $j=\pm 2$.

Now, it suffices to show that $f_2(s,t)$ and $f_{-2}(s,t)$ do not vanish simultaneously.
By the definition of $C$, we have $s\neq 0, \pm 1$.
It follows from $t\neq 0$ and $\phi(s,t)=0$ that $s^2\pm s-1 \neq 0$ since $s^2+s^{-2}-3=(s-s^{-1}+1)(s-s^{-1}-1)$.
Finally, by a computer calculation, we can check that $\res_s(g_2(s), g_{-2}(s)) \in\Z$ is non-zero, namely, they have no common zeros.
\end{proof}

\begin{figure}[h]
 \centering
 \includegraphics[width=0.7\textwidth]{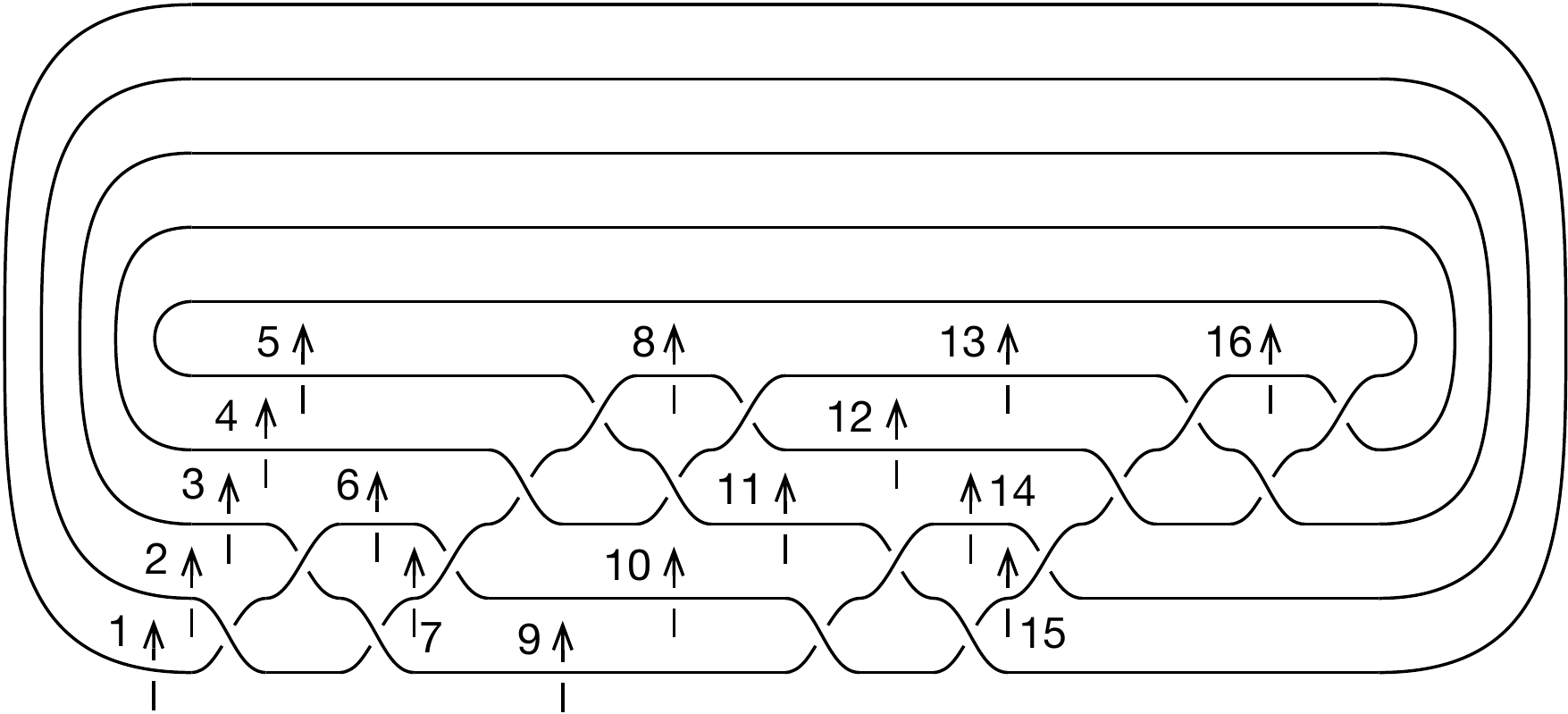}
 \caption{Generators $x_1,x_2,\dots,x_{16}$ of $\pi_1(E(K_0))$.}
 \label{fig:generator}
\end{figure}

\begin{example}
When $(s,t)=(i,\frac{-5+\sqrt{5}}{2})$, the Reidemeister torsion $\tau_{\rho_{s,t,u}}(E(K_0))$ is equal to
\begin{align*}
\frac{1}{16}\left(\left((6765+7695 i)+(6383+2015 i) \sqrt{5}\right) u^2+(73122+21422 \sqrt{5}) \right.\\
\left. +\left((6765-7695 i)+(6383-2015 i) \sqrt{5}\right)u^{-2} \right)
\end{align*}
Also, when $(s,t)=(2,\frac{5-\sqrt{105}}{8})$, $\tau_{\rho_{s,t,u}}(E(K_0))$ is equal to
\begin{align*}
\left((8505120805-233834087 \sqrt{105}) u^2+16 (633637427 \sqrt{105}+6409291542) \right.\\
\left. +(58019711838 \sqrt{105}+594558745850)u^{-2} \right)
/6291456.
\end{align*}
\end{example}

\begin{remark}
Let $s_\pm=\frac{1\pm\sqrt{5}}{2}$.
Then $s_\pm^{\pm 1}$ are the roots of $s^2-s^{-2}-3$ and we have four reducible representations $\rho_{s_\pm^{\pm 1},0,u}$.
First, $\rho_{s_\pm^{-1},0,u}$ is equivalent to $\rho_{s_\pm,0,u}$.
Next, as elements of $X(E(K_0))$, they are independent of $u$ since
\[
P_u^{-1}
\begin{pmatrix}
 s & 0 \\
 0 & s^{-1}
\end{pmatrix}
P_u
=
\begin{pmatrix}
 s & 1-u^{-2} \\
 0 & s^{-1}
\end{pmatrix}.
\]
Moreover, for any $u \in \C^\times$, we have
\[
\tau_{\rho_{s_\pm,0,u}}(E(K_0)) = 256 (8222 \mp 3677 \sqrt{5}).
\]
\end{remark}

Finally, let us prove Corollary~\ref{cor:RT_infinite} stated at the end of Section~\ref{sec:Intro}.

\begin{proof}[Proof of Corollary~\ref{cor:RT_infinite}]
First, the longitude $\lambda$ of $K_0$ is written as
\[
\lambda=x_2 x_7^{-1} x_1 x_{12} x_5^{-1} x_{11} x_{16}^{-1} x_6^{-1} x_{10} x_{15}^{-1} x_9 x_4 x_{13}^{-1} x_3 x_8^{-1} x_{14}^{-1}.
\]
By a computer calculation, we have $\rho_{s,t,u}(\lambda)_{11} = f(s,t)/g(s)$.
Then the resultant $\res_t(g(s)L-f(s,t), \phi(s,t))$ is the product of four factors $s$, $s\pm 1$, and $f_C(L,s)$, where $f_C(L,s)$ is a polynomial written after Corollary~\ref{cor:RT_infinite}.
Since $s\neq 0, \pm 1$, we conclude that $\rho_{s,t,u}(x)_{11}=s$ and $\rho_{s,t,u}(\lambda)_{11}=L$ must satisfy $f_C(L,s)=0$.

Now, we have a 2-bridge knot such that $f_C(L,M)$ and $A_K(M,L)$ have a common zero $(L_0,M_0)$ with $L_0,M_0 \neq 0$ and $\{L_0,M_0\}\not\subset\{1,-1\}$.
It follows from the definition of the $A$-polynomial and \cite[Theorem~3.1]{CoLo96} that there exist $\{\rho_u\}_u \subset C$ corresponding to $(L_0,M_0)$ and $\rho_K \in X(E(K))$ corresponding to $(M_0,L_0)$.
Moreover, $\{L_0,M_0\}\not\subset\{1,-1\}$ implies that $\rho_u$ coincides with $\rho_K$ on the boundary torus, and thus we have a family $\{\rho_u\ast\rho_K\}_u$ of representations of $\pi_1(\Sigma(K_0,K))$.
By the multiplicativity of the Reidemeister torsion and Theorem~\ref{thm:vary_conti}, the set $\RT(\Sigma(K_0,K))$ is an infinite set.
\end{proof}


\end{document}